\documentclass[10pt,reqno]{amsart}
  \usepackage{amssymb}
  \usepackage{bm}
  \usepackage{geometry}
  \usepackage[mathscr]{euscript}
  \usepackage{esvect}
  \usepackage{hyperref}
  \usepackage{paralist}
  \usepackage{tikz}
	  \usetikzlibrary{matrix,arrows}
      \tikzset{>=latex}
	\usepackage{amsthm}
    \theoremstyle{plain}
      \newtheorem{theorem}{Theorem}
      \newtheorem*{theorem*}{Theorem}
        \newtheorem*{maintheorem}{Theorem~\ref{thm:main theorem}}
      \newtheorem{lemma}[theorem]{Lemma}
      \newtheorem{proposition}[theorem]{Proposition}
      \newtheorem{corollary}[theorem]{Corollary}
      \newtheorem{conjecture}[theorem]{Conjecture}
    \theoremstyle{definition}
      \newtheorem{definition}[theorem]{Definition}
    \theoremstyle{remark}
      \newtheorem*{remark}{Remark}
      \newtheorem{example}[theorem]{Example}
  
	\numberwithin{equation}{section}
	\numberwithin{theorem}{section}
    
  \DeclareMathOperator{\Frac}{Frac}
  
  \newcommand{\blp}{\Bigg(}
  
  \newcommand{\brp}{\Bigg)}
  \newcommand{\lb}{\left[}
  \newcommand{\rb}{\right]}
  
  \newcommand{\cE}{\mathcal E}
  \newcommand{\ce}{\mathcal E}
  \newcommand{\cA}{\mathcal A}
  \newcommand{\vS}{\vec s}
  \newcommand{\vs}{\vec s}

\begin{document}

\title{A Binomial Laurent Phenomenon Algebra Associated to the Complete Graph}

\author[Gastineau]{Stella S. Gastineau}
\email{\textcolor{blue}{\href{mailto:stellasuegastineau@gmail.com}{stellasuegastineau@gmail.com}}}
\address{Department of Mathematics, University of Michigan, Ann Arbor, MI 48109}

\author[Moreland]{Gwyneth Moreland}
\email{\textcolor{blue}{\href{mailto:gwynm@umich.edu}{gwynm@umich.edu}}}
\address{Department of Mathematics, University of Michigan, Ann Arbor, MI 48109}

 \thanks{The authors would like to thank their research advisor, Thomas Lam, for the help 
 and guidance he has provided throughout their research.}
 \thanks{The authors were funded by NSF grant DMS-1160726.}

\begin{abstract}
In this paper we find the exchange graph of $\mathcal{A}({\tau_n})$, the rank $n$
binomial Laurent phenomenon algebra associated to the complete graph $K_n$. More specifically, 
we prove that the exchange graph is isomorphic to that of $\mathcal{A}(t_n)$, rank $n$
linear Laurent phenomenon algebra associated to the complete graph which is discussed in
\cite{LP2}.
\end{abstract}
\maketitle

\section{Introduction}\label{sec:introduction}
Cluster algebras were first introduced by Sergey Fomin and Andrei Zelevinsky \cite{fomin}. 
When a cluster algebra is of finite type, that is, there are finitely many seeds, then its
combinatorial structure can be understood through polytopal complexes
called the {\em generalized associahedra} \cite{CFZ,FZ}. 
Cluster algebras, which only have binomial exchange polynomials, were later generalized to
Laurent phenomenon (LP) algebras by Thomas Lam and Pavlo Pylyavskyy \cite{LP1} to
include exchange polynomials with arbitrarily many monomials. 

Suppose that $R$ is a coefficient ring with unique factorization over $\mathbb{Z}$, and let
$\mathcal{F}$ be the rational function field over $n$ independent variables over $\Frac(R)$.
A Laurent phenomenon algebra is a subring $\mathcal{A}\subset
\mathcal{F}$ paired with a collection of seeds $t=(\mathbf{X},\mathbf{F})$ where
$\mathbf{X}=\{X_1,X_2,\dots,X_n\}\subset\mathcal{F}$ are cluster variables and 
$\mathbf{F}\subset R[X_1,X_2,\dots,X_n]$ are exchange polynomials (these both must satisfy some
conditions which will be given later). These seeds are connected by a process called mutation where
one cluster variable is replaced by a new cluster variable satisfying the relation
  $$(\text{old variable})\cdot(\text{new variable})=\text{exchange Laurent polynomial}.$$
A more complete overview of what a seed is and how they mutate is given in Section~\ref{sec:background}.

The {\em exchange graph}
of an LP algebra is the graph whose vertex set is the collection of seeds in the 
algebra and whose edges correspond to being able to mutate one seed to get the other. 
The {\em cluster complex} of a LP algebra is the simplicial complex with the collection of cluster
variables as its base set and faces equal to the clusters of the LP algebra. Zelevinsky made an
observation on the ``striking similarity'' 
between cluster complexes of {\em nested complexes}, studied by Feichtner and Sturmfels
in \cite{FS,Pos}, and cluster complexes
\cite{Zel}, but specific cluster complexes which were indeed nested complexes were not able to be found.
One goal of Lam and Pylyavskyy was to find such cluster complexes which they succeeded in doing in the form of Graph LP algebras \cite{LP2}.

\subsection{Graph LP algebras}
Let $\Gamma$ be a directed graph on $[n]$, and define the seed $t_\Gamma$ with
cluster variables $\{X_1,\dots,X_n\}$ and exchange polynomials $F_i=A_i+\sum_{i\to j} X_j$, where
$i\to j$ denotes an edge in $\Gamma$. For $\Gamma=K_n$ we will simply write $t_n$.
Then the LP algebra $\mathcal{A}(t_\Gamma)$
generated by the initial seed $t_\Gamma$ would be the {\em Laurent phenomenon algebra associated
to $\Gamma$}. A full description of $\cA(t_{\Gamma})$ for any digraph $\Gamma$, including a full 
description of its cluster complex and exchange graph, 
can be found in \cite[Theorem~1.1 and Theorem~5.1]{LP2}. 
One of highlights of this description is
that the cluster complex of $\cA(t_{\Gamma})$ is the {\em extended nested complex} of $\Gamma$, confirming
the observation of Zelevinsky. 

\begin{example}\label{example: graph Gamma}
Consider the digraph $\Gamma$ on $[5]$ with edges $1\to2$, $2\to1$, $2\to3$, $2\to5$,
$3\to2$, $4\to1$, $4\to3$, $4\to5$, $5\to3$, and $5\to4$ as shown in Figure~\ref{fig:graph lp}.
Then the initial seed $t_\Gamma$ associated with this $\Gamma$ is
  $$\begin{aligned}t_\Gamma&=\{(X_1,A_1+X_2),(X_2,A_2+X_1+X_3+X_5),(X_3,A_3+X_2), \\
  	&\qquad\qquad(X_4,A_4+X_1+X_3+X_5),(X_5,A_5+X_3+X_4)\}.\end{aligned}$$
\end{example}

\begin{figure}
\begin{tikzpicture}[scale=1.25,shorten >=2pt]
\tikzset{
    every node/.style={
        circle,
        draw,
        solid,
        fill,
        inner sep=0pt,
        minimum width=4pt
    }
}
\draw (72:1) node[label=above:$1$] {} -- (144:1) node[label=above:$2$] {};
\draw (144:1) -- (216:1) node[label=below:$3$] {};
\draw[<-] (216:1) -- (288:1) node[label=below:$4$] {};
\draw (288:1) -- (360:1) node[label=right:$5$] {};
\draw[->] (360:1) -- (216:1);
\draw[->] (288:1) -- (72:1);
\draw[->] (144:1) -- (360:1);
\end{tikzpicture}
\caption{The example graph $\Gamma$.}
\label{fig:graph lp}
\end{figure}
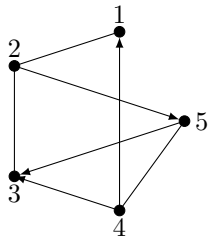

\begin{example}
\label{example:complete graph}
Consider the complete graph $K_n$ on $[n]$ and let $\cA(t_{n})$ be the 
normalized LP algebra associated to $K_n$. Then the exchange graph of $\mathcal{A}(t_{n})$
can be completely described as follows \cite{LP2}:
\begin{itemize}
\item 
The seeds in $\cA(t_{n})$ are in bijection with activation sequences, 
which are simply ordered subsets of $[n]$.
\item
The cluster variables that are not initial cluster variables correspond with 
non-empty unordered subsets of $[n]$.
\item 
Mutating the seed $t_{n}^{\vS}$ can result in three types of seeds $t_{n}^{\vS'}$:
  $$\vS'=\begin{cases}
    (s_1,\dots,s_k,\sigma) & \text{if mutating at some $\sigma\notin\vS$,} \\
    (s_1,\dots,s_{k-1}) & \text{if mutating at $s_k\in\vS$,} \\
    (s_1,\dots,s_{i-1},s_{i+1},s_{i},\dots,s_k) & \text{if mutating at $s_i\in\vS$.}
    \end{cases}$$
\end{itemize}
The cluster complex of $\mathcal{A}(t_{n})$ is the extended nested complex of $K_n$ as stated above.
\end{example}

\subsection{Binomial graph LP algebras}\label{subsec:binomial graph lp algebras}
In the previous subsection we defined a family of initial seeds 
associated with digraphs, but we can also define another family of initial seeds
that can be associated to the same digraphs. Let $\Gamma$ be a digraph on $[n]$ and consider the seed
$\tau_\Gamma$ with cluster variables $\{X_1,\dots,X_n\}$ and exchange polynomials 
$F_i=A_i+\prod_{i\to j}X_j$, where $i\to j$ denotes an edge in $\Gamma$. For $\Gamma=K_n$ we will simply
write $\tau_n$. Note that the exchange polynomials
in $\tau_{\Gamma}$ are binomials, and so we will call the LP algebra $\cA(\tau_{\Gamma})$ generated by
the initial seed $\tau_\Gamma$ the {\em binomial Laurent phenomenon algebra associated to $\Gamma$}.

\begin{example}
Recall the digraph $\Gamma$ as shown in Figure~\ref{fig:graph lp} and used in 
Example~\ref{example: graph Gamma}. The initial binomial seed $\tau_\Gamma$ associated to $\Gamma$ is
  $$\begin{aligned}\tau_\Gamma&=\{(X_1,A_1+X_2),(X_2,A_2+X_1X_3X_5),(X_3,A_3+X_2), \\
  	&\qquad\qquad(X_4,A_4+X_1X_3X_5),(X_5,A_5+X_3X_4)\}.\end{aligned}$$
\end{example}

In this paper we will completely describe the specific case of $\cA(\tau_{n})$ the normalized binomial
LP algebra associated to the complete graph $K_n$. In \S\ref{sec:recursive} we define the exchange
polynomials of $\cA(\tau_{n})$ which we will prove in \S\ref{seeds and mutations}. We will also
describe the cluster variables of $\cA(\tau_{n})$ in \S\ref{seeds and mutations}, and
we will prove how seeds mutate in $\cA(\tau_{n})$. Our main theorem is

\begin{theorem}
\label{thm:main theorem}
Let $\mathcal{A}(\tau_{n})$ be the normalized LP algebra 
generated by initial seed $\tau_{n}$. Similarly, let $\mathcal{A}(t_{n})$ 
be the normalized LP algebra generated by the initial $t_{n}$. Then the respective exchange graphs
of $\mathcal{A}(\tau_{n})$ and $\mathcal{A}(t_{n})$ are isomorphic.
\end{theorem}

The description of $\cA(t_n)$ we gave in Example~\ref{example:complete graph} allows us to prove 
Theorem~\ref{thm:main theorem} in three parts:
\begin{itemize}
\item The seeds in $\cA(\tau_{n})$ are in bijection with activation sequences [Corollary~\ref{cor:seed mutation}].
\item The cluster variables that are not initial cluster variables are in bijection with
non-empty unordered subsets of $[n]$ [Corollary~\ref{cor:exchange polynomials}].
\item Mutating the seed $\tau_{n}^{\vS}$
can result in three types of seeds $\tau_{n}^{\vS'}$ [Theorem~\ref{thm:mutating seeds}]:
  $$\vS'=\begin{cases}
    (s_1,\dots,s_k,\sigma) & \text{if mutating at some $\sigma\notin\vS$,} \\
    (s_1,\dots,s_{k-1}) & \text{if mutating at $s_k\in\vS$,} \\
    (s_1,\dots,s_{j-1},s_{j+1},s_{j},\dots,s_k) & \text{if mutating at $s_j\in\vS(k-1)$.}
    \end{cases}$$
\end{itemize}
Therefore if we prove this we have also proven that the cluster complex of $\cA(\tau_n)$ is 
also the extended nested complex of $K_n$.

\section{Background on LP algebras}\label{sec:background}
An extensive overview of Laurent phenomenon algebras can be found in \cite{LP1} and \cite{LP2}; therefore,
we will only give a brief description of LP algebras as to act as a reference for this paper.

\subsection{Seeds}
Let $R$ be a coefficient ring with unique factorization over $\mathbb{Z}$. Consider the rational
function field $\mathcal{F}$ over $n$ independent variables over the field of fractions
$\Frac(R)$. We will this $\mathcal{F}$ the {\em ambient field}.

\begin{definition}
\label{definition:seeds}
A {\em seed} of rank $n$ in $\mathcal{F}$ is an ordered pair $(\mathbf{X},\mathbf{F})$ consisting
of $\mathbf{X}=\{X_1,\dots,X_n\}$ a transcendence basis of $\mathcal{F}$ and $\mathbf{F}=
\{F_1,\dots,F_n\}$ a collection of auxiliary polynomials in $\mathcal{P}=R[X_1,\dots,X_n]$ satisfying the
following:
  \begin{description}
  \item[LP1] Each $F_i$ is irreducible in $\mathcal{P}$ and is not divisible by any variable $X_j$.
  \item[LP2] Each $F_i$ does not involve the variable $X_i$.
  \end{description}
The set $\mathbf{X}$ is called the {\em cluster}, and its elements are called the {\em cluster variables}.
The elements of $\mathbf{F}$ are called the {\em exchange polynomials}. Although these sets are unordered,
it will often be convenient to think of them and write them as ordered.
\end{definition}

For a seed $(\mathbf{X},\mathbf{F})$ we denote the Laurent polynomial ring $R[X_1^{\pm 1},\dots,X_n^{\pm
1}]$ as $\mathcal{L}=\mathcal{L}(\mathbf{X},\mathbf{F})$. Now we can define a collection of 
polynomials $\hat{\mathbf{F}}=\{\hat F_1,\dots,\hat F_n\}\subset \mathcal{L}$ so that they satisfy
the following:
  \begin{enumerate}
  \item Each $\hat F_i=X_1^{a_1}\cdots\widehat{X_i}\cdots X_n^{a_n}F_i$ for some
  $a_1,\dots,\widehat{a_i},\dots,a_n\in\mathbb{Z}_{\leq0}$.
  \item Each $\hat F_i|_{X_j\leftarrow F_j/X}$ is in $R[X_1^{\pm1},\dots,X_{j-1}^{\pm1},
  X^{\pm1},X_{j+1}^{\pm1},\dots,X_n^{\pm1}]$ and is not divisible by $F_j$ in 
    $$R[X_1^{\pm1},\dots,X_{j-1}^{\pm1},X^{\pm1},X_{j+1}^{\pm1},\dots,X_n^{\pm1}].$$    
  \end{enumerate}
The Laurent polynomials in $\hat{\mathbf{F}}$ are uniquely defined by $\mathbf{F}$ and likewise
$\hat{\mathbf{F}}$ uniquely defines $\mathbf{F}$. Note that a lower bound for each $a_j$ above is minus
the maximum number of times one can factor $F_{i}|_{X_j\leftarrow0}$ by $F_j$.

\subsection{The mutation process} 
Suppose we have a seed $t=(\mathbf{X},\mathbf{F})$. Then for 
every $i\in[n]$, we can obtain a new seed written $t'=(\mathbf{X}',\mathbf{F}')
=\mu_i(\mathbf{X},\mathbf{F})$ where the cluster variables of $t'$ are
  \begin{equation}
  \label{equation:mutating variables}
  \mu_i(X_j)=X'_j=\begin{cases}X_j&\text{if $j\neq i$,}\\ \hat F_i/X_i&\text{if $j=i$.}\end{cases}
  \end{equation}
The exchange polynomials $F'_j\in\mathcal{L}'$ are obtained from the $F_j$. If $F_j$ does not
depend on $X_i$, then $F'_j=F_j$ and we consider $F'_j$ as an element of $\mathcal{L}'$.
Otherwise, we define $G_j$ by
  $$G_j=F_j\Bigg|{}_{X_i\leftarrow \frac{\hat F_i|_{X_j\leftarrow0}}{X'_i}}.$$
Note that using $\hat F_i$ with $X_j=0$ in our substitution guarantees that
our final polynomial will not depend on $X_j$. Next we define $H_j$ as $G_j$ with all common
factors in $R[X_1,\dots,\widehat{X_i},\dots,\widehat{X_j},\dots,X_n]$ with $\hat
F_i|_{X_j\leftarrow0}$ divided out. Lastly, we have that $\mu_i(F_j)=F'_j=MH_j$ where $M$ is a Laurent
monomial in $X_1,\dots,\widehat{X'_j},\dots,X_n$ such that $F'_j$ satisfies the condition 
{\bf LP1} and is not divisible by any variable in $P'$. That is, $M$ clears the denominators of $H_j$.

\begin{remark}
Our definition only defines $F'_j$ up to a unit in $R$. However, in our results,
we give an an explicit choice of $F'_j$ removing the ambiguity in a consistent way. Also,
Proposition~2.9 in \cite{LP1} ensures us that mutation at $i$ produces a valid seed, while
Proposition~2.10 tells us that $\mu_i(\mu_i(\mathbf{X},\mathbf{F}))=(\mathbf{X},\mathbf{F})$,
so the edges of the exchange graph are not directed.
\end{remark}

\subsection{Laurent phenomenon algebras} 
A {\em Laurent phenomenon algebra} $(\mathcal A,\{(\mathbf{X},\mathbf{F})\})$ is a subring
$\mathcal{A}\subset\mathcal{F}$ and a collection of seeds $\{(\mathbf{X},\mathbf{F})\}\subset
\mathcal{F}$. The algebra $\mathcal{A}\subset \mathcal{F}$ is generated over $R$ by all
exchange variables of seeds $\{(\mathbf{X},\mathbf{F})\}$. Given a seed in $t$ in $\mathcal{F}$,
we shall denote $\mathcal{A}(t)$ to be any LP algebra that has $t$ as a seed and say that it is
generated by the {\em initial seed} $t$.

We shall say that two seeds $t=(\mathbf{X},\mathbf{F})$ and $t'=(\mathbf{X}',\mathbf{F}')$ are
{\em equivalent} provided that $X_i/X'_i$ and $F_i/F'_i$ are units in $R$ for each $i\in[n]$, where
$F_i$ and $F'_i$ are both considered elements of the ambient field $\mathcal{F}$. We shall say an
LP algebra $\cA$ is {\em normalized} if whenever two seeds $t$ and $t'$ are equivalent, we have $t=t'$.
If $\cA$ is normalized, we shall say that it is of {\em finite type} if it has finitely many seeds. The
LP algebras $\cA(t_n)$ and $\cA(\tau_n)$ (see Section \ref{sec:introduction}) studied 
in this paper will be normalized and of finite type.

The {\em exchange graph} on an LP algebra $(\mathcal{A},\{(\mathbf{X,F})\})$ is a connected graph with
vertex set equal to the seeds of $\{(\mathbf{X,F})\}$ and the edges given by mutations; that is,
$\{t,t'\}$ is an edge in the exchange graph if there exists an $i$ such that $\mu_i(t)=t'$.

\section{Set-Up}\label{sec:recursive}
An {\em activation sequence} $\vec s=(s_1,\dots,s_k)\subset[n]$ of length $k$ is an ordered subset of $[n]$. 
Activation sequences can also be thought of as a partial permutation of $[n]$, but we prefer
to think of them as ordered subsets of $[n]$. If $\vec s$
is an activation sequence of length $k$, then for every $r\in[k]$, we denote $\vec s(r)\subset \vec
s$ as the {\em subactivation sequence} of length $r$ whose ordered elements are the first 
$r$ ordered elements of $\vec s$; that is,
  $$\vec s(r)=(s_1,\dots,s_r)\subset (s_1,\dots,s_r,\dots,s_k)\subset[n].$$
The {\em underlying set} $s$ of an activation sequence $\vec s$ just refers to the 
forgetful (unordered) copy of $\vs$; that is,
  $$s=\{\sigma\in\vec s\}.$$

Recall that $\tau_n=(\mathbf{X,F})$ is the binomial Laurent phenomenon algebra associated to the complete graph.
If $\vs\subset[n]$ is an activation sequence of length $k$, then we will denote 
$\tau_n^{\vs}=\mu_{\vs}(\tau_n)=\big(\mathbf{Y}^{\vs},\mathbf{E}^{\vs}\big)$ to be the seed obtained by mutating
$\tau_n$ at $s_1,s_2,\dots,s_k$ in that order
  $$\tau_n^{\vec s}=\mu_{s_k}\mu_{s_{k-1}}\cdots\mu_{s_1}(\tau_n).$$
We say that $\tau_n^{\vec s}$ is the {\em seed generated by $\vec s$}. 
We will write $Y_{\ell}^{\vs}$ to denote the cluster variables $\mathbf{Y}^{\vs}$ of $\tau_n^{\vs}$.
Since these cluster variables come in the form given by \eqref{equation:mutating variables}, we know we
can simplify this notation to say that recursively denote
$Y_{\vs}=\mu_{s_k}(Y_{s_k}^{\vs(k-1)})$ where $Y_{(s_1)}=\mu_{s_1}(X_{s_1})$.
It is not difficult to see that we can now write the cluster variables $\mathbf{Y}^{\vs}$ as either
$\{Y_{\ell}^{\vs}\mid \ell\in[n]\}$ or as $\{X_\sigma\mid \sigma\notin\vs\}\cup\{Y_{\vs(r)}\mid 1\leq r\leq k\}$.
Throughout this paper, we will switch between these two notations depending on which is the most convenient at the time.

Our first goal is to explicitly describe the seed associated to each activation sequence $\vs\subset[n]$. In
Definition~\ref{defn:recursive structure} we give a recursive set of polynomials $\bm{\cE}^{\vs}$
in our polynomial ring $R[\mathbf{Y}^{\vs}]$
which we will prove to be the exchange polynomials of $\tau_n^{\vs}$ in Theorem~\ref{thm:exchange polynomials}.
This will allow us to conclude that the seeds of $\mathcal{A}(\tau_n)$ are in bijection with activation sequences.
With an explicit description of the exchange polynomials, we will also be able to give a formula for the 
cluster variables $\mathbf{Y}^{\vs}$ in terms of the initial cluster variables $\mathbf{X}$ in 
Corollary~\ref{cor:exchange polynomials}, where it will follow that cluster variables of $\mathcal{A}(\tau_n)$ are
in bijection with non-empty ordered subsets of $[n]$ (i.e., the underlying sets of the activation sequences).

\begin{definition}
\label{defn:recursive structure}
Let $\vec s\subset [n]$ be an activation sequence of length $k$.
Then for all $\sigma\notin\vec s$, let
  \begin{equation}
  \label{eqn:E notAC}
  \cE^{\vec s}_{\sigma}=\blp\prod_{r=1}^{k}A_{s_r}\brp
  \blp\prod_{{\rho\notin \vec s} \atop {\rho\neq\sigma}}X_\rho\brp+A_\sigma Y_{\vec s}. 
  \end{equation}
For each $s_i\in\vec s$ we can define
  \begin{equation}
  \label{eqn:C AC}
  C^{\vec s}_{s_i}=\prod_{r=i+2}^{k}Y_{\vec s(r)}^{2^{r-i-1}}.
  \end{equation}
Then let  
  \begin{equation}
  \label{eqn:P k AC}
  P^{\vec s}_{s_k}=\blp\prod_{r=1}^{k-1}A_{s_r}\brp\blp\prod_{\rho\notin\vec s}X_\rho\brp
  +A_{s_k}Y_{\vec s(k-1)},
  \end{equation}
  \begin{equation}
  \label{eqn:E k AC}
  \cE^{\vec s}_{s_k}=P^{\vec s}_{s_k}.
  \end{equation}
For every $s_i\in\vec s(k-1)=(s_1,\dots,s_{k-1})$ let
  \begin{equation}
  \label{eqn:P AC}
  P^{\vec s}_{s_i}=\blp\prod_{r=1}^{i-1}A_{s_r}\brp\blp\prod_{\rho\notin\vec s}X_\rho\brp
  \blp\prod_{r=i+1}^{k} P^{\vec s}_{s_r}\brp+A_{s_i}Y_{\vS(i-1)}Y_{\vS(i+1)}C^{\vS}_{s_i}
  \end{equation}
  \begin{equation}
  \label{eqn:E AC}
  \cE^{\vec s}_{s_i}=\blp\prod_{r=1}^{i-1}A_{s_r}\brp^2\blp\prod_{\rho\notin\vec s}X_\rho\brp^2
  \blp\prod_{r=i+2}^{k} P^{\vec s}_{s_r}\brp^2+Y_{\vS(i-1)}Y_{\vS(i+1)}C^{\vS}_{s_i}.
  \end{equation}
Note that above when $i=k-1$ the product
  $$\prod_{r=k+1}^kP^{\vec s}_{s_r}$$
which does not make sense. In this case (and all other cases where bottom limit of a product is greater
than the top limit) it is understood that the product is empty and equals 1. 
With \eqref{eqn:E AC}, we can write \eqref{eqn:P AC} as
  \begin{equation}
  \label{eqn:P alt AC}
  \begin{aligned}
  P^{\vec s}_{s_i}
  &=\blp\prod_{r=1}^{i-1}A_{s_r}\brp\blp\prod_{\rho\notin\vec s}X_\rho\brp
  \blp\prod_{r=i+2}^{k} P^{\vec s}_{s_r}\brp \\
  &\qquad\cdot\lb \blp\prod_{r=1}^{i}A_{s_r}\brp
  \blp\prod_{\rho\notin\vec s}X_\rho\brp\blp\prod_{r=i+2}^{k} P^{\vec s}_{s_r}\brp
  + A_{s_{i+1}}Y_{\vS(i)}Y_{\vS(i+2)}C^{\vS}_{s_{i+1}}\rb\\
  &\qquad+A_{s_i}Y_{\vS(i-1)}Y_{\vS(i+1)}C^{\vS}_{s_i} \\
  &=A_{s_{i+1}}Y_{\vS(i)}Y_{\vS(i+2)}C^{\vS}_{s_{i+1}}\blp\prod_{r=1}^{i-1}A_{s_r}\brp
  \blp\prod_{\rho\notin\vec s}X_\rho\brp\blp\prod_{r=i+2}^{k} P^{\vec s}_{s_r}\brp
  +A_{s_i}\cE^{\vec s}_{s_i}
  \end{aligned}
  \end{equation}
which will be advantageous to use at times. We will denote the set of all $\cE^{\vec s}_{\ell}$ for
$\ell\in[n]$ as $\bm{\cE}^{\vec s}=\{\cE^{\vec s}_{\ell}\mid \ell\in[n]\}$.
\end{definition}

Now we will begin to look at the structure of the polynomials in 
Definition~\ref{defn:recursive structure}. In order to prove that $\bm{\cE}^{\vs}$ are indeed
the exchange polynomials of the seed $\tau_n^{\vs}$, we will need to understand the polynomials
$\hat{\bm{\cE}}{}^{\vs}=\{\hat{\cE}^{\vs}_1,\dots,\hat{\cE}^{\vs}_n\}$ defined by
  \begin{equation}
  \label{eqn:denom defn}
  \hat{\cE}^{\vs}_\ell=\blp\prod_{{\rho\notin\vs}\atop{\rho\neq\ell}}X_\rho^{a_\rho}\brp
    \blp\prod_{{s_j\in\vs}\atop{s_j\neq\ell}}Y_{\vs(j)}^{b_j}\brp\cE^{\vs}_\ell
  \end{equation}
for some $a_\rho$ and $b_j$ for all $\ell\in[n]$. 
In order to find these $a_\rho$ and $b_j$ for any given $\ell\in[n]$, we will need to respectively
understand how $\cE_{\rho}^{\vs}$ and $\cE_{s_j}^{\vs}$ can be factored from $\cE_{\ell}^{\vs}$.


\begin{lemma}
\label{lemma:denom lemma}
Suppose that $\vs\subset[n]$ is an activation sequence of length $k$. If $i+2\leq j\leq k$, then 
there are exactly $2^{j-i-1}$ factors of $\cE_{s_j}^{\vs}$ in 
  \begin{equation}
  \label{eqn:denom lemma}
  \cE_{s_i}^{\vs}|_{Y_{\vs(j)}\leftarrow0}=
    \blp\prod_{r=1}^{i-1}A_{s_r}\brp^2\blp\prod_{\rho\notin\vs}X_\rho\brp^2\blp\prod_{r=i+2}^kP_{s_r}^{\vs}
     \Big|_{Y_{\vs(j)}\leftarrow 0}\brp^2.
  \end{equation}
\end{lemma}

\begin{proof}
First we will prove by induction on $r$ that for $j\leq k-2$ and all $j+2\leq r\leq k$, 
$\cE_{s_j}^{\vs}$ can not be factored
from $P_{s_r}^{\vs}|_{Y_{\vs(j)}\leftarrow0}$. For $r=k$, we have that
$P_{s_k}^{\vs}|_{Y_{\vs(j)}\leftarrow0}=P_{s_k}^{\vs}$ does not depend
on $Y_{\vs(j+1)}$ but $\cE_{s_j}^{\vs}$ does, and so $\cE_{s_j}^{\vs}\neq P_{s_k}^{\vs}$. Then
since $P_{s_k}^{\vs}=\cE_{s_k}^{\vs}$ which is irreducible, we have that we cannot factor
$P_{s_r}^{\vs}$ by $\cE_{s_j}^{\vs}$.
Then, looking at \eqref{eqn:P alt AC} for any $j+2\leq r\leq k-1$ and noting that once
again $\cE_{s_r}^{\vs}$ is irreducible, we get that $\cE_{s_k}^{\vs}$ does not factor out of
$P_{s_r}^{\vs}|_{Y_{\vs(j)}\leftarrow0}=P_{s_r}^{\vs}$ by induction.

For the cases of $P_{s_{j+1}}^{\vs}$ and $P_{s_j}^{\vs}$, by \eqref{eqn:P AC} we have that
  \begin{equation}\label{eqn:denom P j+1}
  P^{\vec s}_{s_{j+1}}\Big|_{Y_{\vs(j)}\leftarrow0}
    =\blp\prod_{r=1}^{j}A_{s_r}\brp\blp\prod_{\rho\notin\vec s}X_\rho\brp
    \blp\prod_{r=j+2}^{k} P^{\vec s}_{s_r}\brp
  \end{equation}
which by the previous paragraph does not have a factor of $\cE_{s_j}^{\vs}$ and 
  \[P_{s_j}^{\vs}\Big|_{Y_{\vs(j)}\leftarrow0}=A_{s_j}\cE_{s_j}\]
which trivially has at most $1$ factor of $\cE_{s_j}$.
Now we will show by induction on $r$ that $P_{s_r}^{\vs}|_{Y_{s_j}\leftarrow0}$ has
$2^{j-r-1}$ factors of $\cE_{s_j}^{\vs}$ for $i+2\leq r\leq j-1$.
From \eqref{eqn:denom P j+1} we can conclude that
  \begin{align*}
  P_{s_{j-1}}^{\vs}\Big|_{Y_{\vs(j)}\leftarrow0}
    & =\Bigg[\blp\prod_{r=1}^{j-2}A_{s_r}\brp\blp\prod_{\rho\notin\vec s}X_\rho\brp
      \blp\prod_{r=j+1}^{k} P^{\vec s}_{s_r}\brp \\
    & \qquad\cdot\lb \blp\prod_{r=1}^{j}A_{s_r}\brp
      \blp\prod_{\rho\notin\vec s}X_\rho\brp\blp\prod_{r=j+1}^{k} P^{\vec s}_{s_r}\brp
      + A_{s_{j}}Y_{\vS(j-1)}Y_{\vS(j+1)}C^{\vS}_{s_{j}}\rb\\
    & \qquad+A_{s_{j-1}}Y_{\vS(i-2)}Y_{\vS(j)}C^{\vS}_{s_{j-1}}\Bigg]_{Y_{s_j}\leftarrow0} \\
    & =A_{s_{j-1}}A_{s_j}\blp\prod_{r=1}^{j-2}A_{s_r}\brp^2\blp\prod_{\rho\notin\vec s}X_\rho\brp^2
      \blp\prod_{r=j+2}^{k} P^{\vec s}_{s_r}\brp \\
    & \qquad\cdot\lb A_{s_{j}}\blp\prod_{r=1}^{j-1}A_{s_r}\brp^2
      \blp\prod_{\rho\notin\vec s}X_\rho\brp\blp\prod_{r=j+2}^{k} P^{\vec s}_{s_r}\brp
      + A_{s_{j}}Y_{\vS(j-1)}Y_{\vS(j+1)}C^{\vS}_{s_{j}}\rb \\
    & =A_{s_{j-1}}A_{s_j}^2\blp\prod_{r=1}^{j-2}A_{s_r}\brp^2\blp\prod_{\rho\notin\vec s}X_\rho\brp^2
      \blp\prod_{r=j+2}^{k} P^{\vec s}_{s_r}\brp\cE_{s_j}^{\vs}
  \end{align*}
which contains exactly $1=2^{j-(j-1)-1}$ factor of $\cE_{s_j}$.
Then for $i+2\leq r< j-1$ we have that
  \[P_{s_r}^{\vs}\Big|_{Y_{\vs(r)}\leftarrow0}=
  \blp\prod_{\ell=1}^{r-1}A_{s_\ell}\brp\blp\prod_{\rho\notin\vec s}X_\rho\brp
  \blp\prod_{\ell=r+1}^{k} P^{\vec s}_{s_\ell}\Big|_{Y_{\vs(j)}\leftarrow0}\brp\]
and so by induction contains exactly
  \[1+\sum_{\ell=r+1}^{j-1}2^{j-\ell-1}=2^{j-r-1}\]
factors of $\cE_{s_j}^{\vs}$ (where the extra $1$ comes from the $P_{s_j}^{\vs}$ term).
Therefore, there are 
  \[2\blp1+\sum_{r=i+2}^{j-1}2^{j-r-1}\brp=2\cdot 2^{j-i-2}=2^{j-i-1}\]
factors of $\cE_{s_j}^{\vs}$ in \eqref{eqn:denom lemma}.
\end{proof}

\begin{proposition}
\label{prop:denom}
Given an activation sequence $\vec s\subset[n]$ of length $k$, if $\sigma\notin\vec s$ then
  \begin{equation}
  \label{eqn:denom notAC}
  \hat{\cE}^{\vec s}_{\sigma}/\cE^{\vec s}_{\sigma}=1
  \end{equation}
and if $s_i\in\vec s$ then 
  \begin{equation}
  \label{eqn:denom AC}
  \hat{\cE}^{\vec s}_{s_i}/\cE^{\vec s}_{s_i}=1/C^{\vec s}_{s_i}
  \end{equation}
\end{proposition}

\begin{proof}

Fix $\vs\subset[n]$ an activation sequence of length $k$.
First note that if $\cE_{i}^{\vs}\in\bm{\cE}^{\vs}$ 
is independent of the cluster $Y_{\ell}^{\vs}\in \mathbf{Y}^{\vs}$ then
$\cE_{i}^{\vs}|_{Y_\ell^{\vs}\leftarrow0}=\cE_{i}^{\vs}$
which is irreducible in $\mathcal{P}$. Therefore the power of $Y_{\ell}^{\vs}$ in 
$\hat{\cE}^{\vec s}_{i}/\cE^{\vec s}_{i}$
is $0$. We will not discuss the proof for \eqref{eqn:denom notAC}
because the arguments follow from the simplicity of $\cE_{\sigma}^{\vs}$ for $\sigma\notin\vs$.
Furthermore, all of the arguments are similar to those used to prove \eqref{eqn:denom AC} which 
is provided below.

Next we will prove equation \eqref{eqn:denom notAC}.
Fix $\sigma\notin\vs$. 
Recall our condition that $\hat{\cE}^{\vS}_{\sigma}
|_{Y_{\ell}^{\vS}\leftarrow\cE^{\vS}_{\ell}/X}$ is in 
$$R[(Y^{\vs}_1)^{\pm1},\dots,(Y^{\vs}_{\ell-1})^{\pm1},X^{\pm1},(Y^{\vs}_{\ell+1})^{\pm1},
\dots,(Y^{\vs}_n)^{\pm1}]$$
and is not divisible by $\cE^{\vS}_{\ell}$ in 
$R[(Y^{\vs}_1)^{\pm1},\dots,(Y^{\vs}_{\ell-1})^{\pm1},X^{\pm1},(Y^{\vs}_{\ell+1})^{\pm1},
\dots,(Y^{\vs}_n)^{\pm1}]$.
Fix $j\in[n]\setminus\vS$ such that $j\neq\sigma$. Making the substitution
  $$\cE^{\vS}_{\sigma}|_{X_j\leftarrow\cE^{\vS}_{j}/X}=
    \blp\prod_{r=1}^{k}A_{s_r}\brp \blp\prod_{{\rho\notin \vec s} \atop {j\neq\rho\neq\sigma}}
    X_\rho\brp\blp\frac{\cE^{\vS}_{j}}{X}\brp+A_\sigma Y_{\vS}$$
we need to factor out $\cE^{\vS}_{j}$ the precise number of times to satisfy these two conditions
(where the well definedness of this number is checked in \cite[Lemma~2.3]{LP1}), and 
minus the number of times we can factor our $\cE^{\vS}_{j}$ is precisely the power of $X_j$ in 
$\hat{\cE}^{\vec s}_{\sigma}/\cE^{\vec s}_{\sigma}$. By its form, 
$\cE^{\vS}_{\sigma}|_{X_j\leftarrow\cE^{\vS}_{j}/X}$ is already in 
$$R[(Y^{\vs}_1)^{\pm1},\dots,(Y^{\vs}_{j-1})^{\pm1},X^{\pm1},(Y^{\vs}_{j+1})^{\pm1},
\dots,(Y^{\vs}_n)^{\pm1}],$$ 
and {\bf LP1}
ensures $\cE^{\vS}_{\sigma}|_{X_j\leftarrow\cE^{\vS}_{j}/X}$ is not divisible by $\cE^{\vS}_{j}$. Therefore
the power of $X_j$ in $\hat{\cE}^{\vec s}_{\sigma}/\cE^{\vec s}_{\sigma}$ is zero. Next, for any
$s_i\in\vS(k-1)$, $\cE_\sigma^{\vS}$ is independent of $Y_{\vS(i)}$ and so the power of $Y_{\vs(i)}$
in $\hat{\cE}^{\vec s}_{\sigma}/\cE^{\vec s}_{\sigma}$ is zero. Lastly, consider the case when $\ell=s_k$.
Making the substitution,
  $$\cE^{\vS}_{\sigma}|_{Y_{\vs}\leftarrow\cE^{\vS}_{s_k}/X}=
    \blp\prod_{r=1}^{k}A_{s_r}\brp \blp\prod_{{\rho\notin \vec s} \atop {\rho\neq\sigma}}
    X_\rho\brp+A_\sigma \blp\frac{\cE^{\vS}_{s_k}}{X}\brp$$
is once again already in 
$R[(Y^{\vs}_1)^{\pm1},\dots,(Y^{\vs}_{s_k-1})^{\pm1},X^{\pm1},(Y^{\vs}_{s_k+1})^{\pm1},
\dots,(Y^{\vs}_n)^{\pm1}]$
and not divisible by $\cE_{s_k}^{\vs}$ by {\bf LP1}, and so the factor of $Y_{\vs}$ in 
$\hat{\cE}^{\vec s}_{\sigma}/\cE^{\vec s}_{\sigma}$ is zero as well. Thus
  $$\hat{\cE}^{\vec s}_{\sigma}/\cE^{\vec s}_{\sigma}=1$$
as was desired.

As you can see in \eqref{eqn:denom notAC} and \eqref{eqn:denom AC}, the only
interesting cases arise when $\ell\in\vs$. So
suppose that $\ell=s_i\in\vs$ and let $s_j\in\vs$ with $i+2\leq j$. 
We can then write $\cE_{s_i}^{\vs}$ as a polynomials of $Y_{\vs(j)}$ with all the other 
cluster variables held ``constant''
  \[\cE_{s_i}^{\vs}=e_0+e_1Y_{\vs(j)}+e_2 Y_{\vs(j)}^2+\cdots.\]
By construction of the polynomial $\hat{\cE}_{s_i}^{\vs}$,
we have see that the power of $Y_j$ in \eqref{eqn:denom AC} is given by
  \[b_j:=-\max\{b\in\mathbb{Z}_{\geq0}\mid\mbox{$(\cE_{s_j}^{\vs})^{b-\alpha}$ can be factored from
    $e_\alpha$ for all $\alpha\in\mathbb{Z}_{\geq 0}$}\}.\]
    Then Lemma~\ref{lemma:denom lemma} gives us that $-2^{j-i-1}\leq b_j\leq 0$ because 
$e_0=\cE_{s_i}^{\vs}|_{Y_{\vs(j)}\leftarrow0}$.
In order to find a better upper bound on $b_j$ (which will eventually by $2^{j-i-1}$),
we will similarly write each $P_{s_r}^{\vs}$ as a polynomial of $Y_{\vs(j)}$
  \[P_{s_r}^{\vs}=p_{r,0}+p_{r,1}Y_{\vs(j)}+p_{r,2}Y_{\vs(j)}^2+\cdots.\]

For all $i+1\leq r\leq k$ and all $\beta\in\mathbb{Z}_{\geq0}$ 
let $\mathcal{C}_{r,k}$ be the collection of
sequences $(\gamma_\ell)=(\gamma_{r+1},\gamma_{r+2},\dots,\gamma_{k})$ over $\mathbb{Z}_{\geq 0}$
such that $\gamma_{r+1}+\cdots+\gamma_k=\beta$.
Then from \eqref{eqn:P AC} we can recursively rewrite each $p_{r,\beta}$ as
 \begin{equation}{\scriptsize\label{eqn:p expand}
  p_{r,\beta} Y_{\vs(j)}^{\beta}=\begin{cases}
  \blp\prod\limits_{\ell=1}^{r-1}A_{s_\ell}\brp\blp\prod\limits_{\rho\notin\vec s}X_\rho\brp
    \blp\sum\limits_{(\gamma_{\ell})\in\mathcal{C}_{r,\beta}} 
    \blp\prod\limits_{\ell=r+1}^{k} p_{\ell,\gamma_\ell}\brp\brp Y_{\vs(j)}^\beta
    & \mbox{$\beta\neq 2^{r-i-1}$} \\
  \blp\prod\limits_{\ell=1}^{r-1}A_{s_\ell}\brp\blp\prod\limits_{\rho\notin\vec s}X_\rho\brp
    \blp\sum\limits_{(\gamma_{\ell})\in\mathcal{C}_{r,\beta}} 
    \blp\prod\limits_{\ell=r+1}^{k} p_{\ell,\beta_\ell}\brp\brp Y_{\vs(j)}^\beta +
    A_{s_r}Y_{\vs(r-1)}Y_{\vs(r+1)}C_{s_r}^{\vs}
    & \mbox{$\beta= 2^{r-i-1}$.}\end{cases}}
  \end{equation}
Similarly, for all $\delta\in\mathbb{Z}_{\geq0}$ let
$\mathcal{B}_{\delta}$ be the collection of sequences
$(\beta_r)=(\beta_{i+2},\dots,\beta_k)$ over $\mathbb{Z}_{>0}$
such that $\beta_{i+2}+\cdots+\beta_{k}=\alpha$. Then from \eqref{eqn:E AC} we can rewrite each $e_\alpha$ as
  \begin{equation}{\scriptsize\label{eqn:e expand}
  e_{\alpha}Y_{\vs(j)}^\alpha
    = \begin{cases}
      \blp\prod\limits_{r=1}^{i-1}A_{s_r}\brp^2\blp\prod\limits_{\rho\notin\vec s}X_\rho\brp^2
      \lb\sum\limits_{\delta=0}^{\alpha}\blp\sum\limits_{(\beta_r)\in\mathcal{B}_{\delta}}
      \sum\limits_{(\beta'_r)\in\mathcal{B}_{\alpha-\delta}}
      \prod\limits_{r=i+2}^{k} p_{r,\beta_r}p_{r,\beta'_r}\brp\rb Y_{\vs(j)}^\alpha
      & \mbox{$\alpha\neq2^{j-i-1}$} \\
      \blp\prod\limits_{r=1}^{i-1}A_{s_r}\brp^2\blp\prod\limits_{\rho\notin\vec s}X_\rho\brp^2
      \lb\sum\limits_{\delta=0}^{\alpha}\blp\sum\limits_{(\beta_r)\in\mathcal{B}_{\delta}}
      \sum\limits_{(\beta'_r)\in\mathcal{B}_{\alpha-\delta}}
      \prod\limits_{r=i+2}^{k} p_{r,\beta_r}p_{r,\beta'_r}\brp\rb Y_{\vs(j)}^\alpha + 
      A_{s_i}Y_{\vS(i-1)}Y_{\vS(i+1)}C^{\vS}_{s_i}
      & \mbox{$\alpha=2^{j-i-1}$.}\end{cases}}
  \end{equation}
Note that if $r=j$ or $r\geq j+2$, then $p_{r,\beta}\neq 0$ for only $\beta=0$.
We also know that $p_{j+1,\beta}\neq 0$ for only $\beta\leq1$. Thus we can
inductively see that if $r\leq j-1$, then $p_{r,\beta}\neq0$
for only $\beta\leq 2^{j-r-1}$. Therefore, we can  inductively
see that $e_\alpha\neq0$ only for $\alpha\leq2^{j-i-1}-2$ and $\alpha=2^{j-i-1}$.
It is not difficult to see that if there are at least $b$ factors of $\ce_{s_j}^{\vs}$ 
in $p_{r,\beta}$ where $i+2\leq j\leq k$ and $0\leq\beta\leq{2^{j-r-1}}$. Therefore, 
we get that by induction each $e_\alpha$ has at least
$2^{j-i-1-\alpha}$ factors of $\ce_{s_j}^{e_j}$. Therefore, $2^{j-i-1}\leq b_j\leq 2^{j-1-1}.$
\end{proof}

\section{Seeds and mutation inside $\cA(\tau_n)$}\label{seeds and mutations}
\subsection{Seeds}
We will now begin to prove that the polynomials $\bm{\cE}^{\vs}$ given in Definition~\ref{defn:recursive structure}
are indeed the exchange polynomials $\mathbf{E}^{\vs}$.

\begin{lemma}
\label{lem:exchange polynomials}
Suppose that $\vec s\subset[n]$ is an activation sequence of length $k+1$ and write it
$\vec s=(s_1,\dots,s_k,\sigma)$. Then for all $i\in[k]$,
  \begin{equation}
  \label{eqn:lem exchange polynomials}
  {Y_{\vec s}}^{2^{i+1}-1}\blp\prod_{r=i+1}^{i}P^{\vec s(k)}_{s_{k-r}}\brp\Bigg|_{X_{\sigma}\leftarrow
  \frac{\cE^{\vec s}_{\sigma}}{Y_{\vec s}}}=\blp\prod_{r=0}^{i}P^{\vec s}_{s_{k-r}}\brp.
  \end{equation}
\end{lemma}

\begin{proof}
We will prove this by induction on $i\in[k]$. We will show the base cases when $i=0$ and $i=1$.
When $i=0$, we can directly compute
  \begin{align*}
  Y_{\vec s}P^{\vec s(k)}_{s_k}\Bigg|_{X_{\sigma}\leftarrow
  \frac{\cE^{\vec s(k)}_{\sigma}}{Y_{\vec s}}}
    & =Y_{\vec s}\lb\blp\prod_{r=1}^{k-1}A_{s_r}\brp\blp\prod_{\rho\notin\vec s}X_\rho\brp 
      X_{\sigma}
      +A_{s_k}Y_{\vec s(k-1)}\rb_{X_{\sigma}\leftarrow
      \frac{\cE^{\vec s(k)}_{\sigma}}{Y_{\vec s}}} \\
    & =\blp\prod_{r=1}^{k-1}A_{s_r}\brp\blp\prod_{\rho\notin\vec s}X_\rho\brp 
      \cE^{\vec s}_{\sigma}+A_{s_k}Y_{\vec s(k-1)}Y_{\vec s} \\
    & =A_{s_k}\blp\prod_{r=1}^{k-1}A_{s_r}\brp^2\blp\prod_{\rho\notin\vec s}X_\rho\brp^2
      +A_{\sigma}Y_{\vec s(k)}\blp\prod_{r=1}^{k-1}A_{s_r}\brp
      \blp\prod_{\rho\notin\vec s}X_\rho\brp+A_{s_k}Y_{\vec s(k-1)}Y_{\vec s} \\
    & \stackrel{\eqref{eqn:E AC}}{=}A_{\sigma}Y_{\vec s(k)}\blp\prod_{r=1}^{k-1}A_{s_r}\brp
      \blp\prod_{\rho\notin\vec s}X_\rho\brp+A_{s_k}\cE^{\vec s}_{s_k} \\
    & \stackrel{\eqref{eqn:P alt AC}}{=}P^{\vec s}_{s_k}
  \end{align*}
proving \eqref{eqn:lem exchange polynomials} for $i=0$. When $i=1$ we similarly have
  {\small\begin{align*}
  {Y_{\vec s}}^3P^{\vec s(k)}_{s_k}P_{\vec s(k),s_{k-1}}\Bigg|_{X_{\sigma}\leftarrow
  \frac{\cE^{\vec s(k)}_{\sigma}}{Y_{\vec s}}}
    & =P^{\vec s}_{s_k}\lb{Y_{\vec s}}^2P^{\vec s(k)}_{s_{k-1}}\rb_{X_{\sigma}\leftarrow
      \frac{\cE^{\vec s(k)}_{\sigma}}{Y_{\vec s}}} \\
    & =P^{\vec s}_{s_k}\lb A_{s_k}Y_{\vec s(k-1)}Y_{\vec s}\cE^{\vec s(k)}_{\sigma}
      \blp\prod_{r=1}^{k-2}A_{s_r}\brp\blp\prod_{\rho\notin\vec s}X_\rho\brp\right. \\
    & \qquad\qquad\left.+A_{s_{k-1}}Y_{\vec s(k-2)}Y_{\vec s(k)}{Y_{\vec s}}^2+A_{s_{k-1}}
      (\cE^{\vec s}_{\sigma})^2\blp\prod_{r=1}^{k-2}A_{s_r}\brp^2
      \blp\prod_{\rho\notin\vec s}X_\rho\brp^2\rb \\
    & \stackrel{\eqref{eqn:E k AC}}{=}P^{\vec s}_{s_k}\lb A_{s_k}
      Y_{\vec s(k-1)}Y_{\vec s}P^{\vec s(k)}_{\sigma}
      \blp\prod_{r=1}^{k-2}A_{s_r}\brp\blp\prod_{\rho\notin\vec s}X_\rho\brp\right. \\
    & \qquad\qquad\left.+A_{s_{k-1}}Y_{\vec s(k-2)}Y_{\vec s(k-1)}\mathcal{C}^{\vec s}_{s_k}+A_{s_{k-1}}
      (P^{\vec s(k)}_{\sigma})^2\blp\prod_{r=1}^{k-2}A_{s_r}\brp^2
      \blp\prod_{\rho\notin\vec s}X_\rho\brp^2\rb \\
    & \stackrel{\eqref{eqn:E AC}}{=}P^{\vec s}_{s_k}\lb A_{s_k}
      Y_{\vec s(k-1)}Y_{\vec s}P^{\vec s}_{\sigma}
      \blp\prod_{r=1}^{k-2}A_{s_r}\brp\blp\prod_{\rho\notin\vec s}X_{\rho}\brp+
      A_{s_{k-1}}\cE^{\vec s}_{s_{k-1}}\rb \\
    & \stackrel{\eqref{eqn:P alt AC}}{=} P^{\vec s}_{s_k}P^{\vec s}_{s_{k-1}}
  \end{align*}}
proving \eqref{eqn:lem exchange polynomials} for $i=1$. Consider $i\in[k]$ with $i \ge 2$ 
and suppose that \eqref{eqn:lem exchange polynomials} holds for all $r\in[i-1]$. To prove
that \eqref{eqn:lem exchange polynomials} holds for $i$ it is sufficient to show that
  $${Y_{\vec s}}^{2^i}P^{\vec s(k)}_{s_{k-i}}\Bigg|_{X_{\sigma}\leftarrow
    \frac{\cE^{\vec s(k)}_{\sigma}}{Y_{\vec s}}}=P^{\vec s}_{s_{k-i}}.$$
From definitions and manipulation similar to that seen above we can conclude that
  \begin{align*}
  {Y_{\vec s}}^{2^i}P^{\vec s(k)}_{s_{k-i}}\Bigg|_{X_{\sigma}\leftarrow
  \frac{\cE^{\vec s(k)}_{\sigma}}{Y_{\vec s}}}
    & =A_{s_{k-i+1}}
      Y_{\vec s(k-i)}Y_{\vec s(k-i+2)}\mathcal{C}^{\vec s}_{s_{k-i+1}}
      \blp\prod_{r=1}^{k-i-1}A_{s_r}\brp
      \blp\prod_{\rho\notin\vec s}X_{\rho}\brp
      \blp\prod_{r=k-(i-2)}^kP^{\vec s(k)}_{s_r}\brp \\
    & \qquad\qquad+\lb{Y_{\vec s}}^{2^i}A_{s_{k-i}}
      \cE^{\vec s(k)}_{s_{k-i}}\rb_{X_{\sigma}\leftarrow
      \frac{\cE^{\vec s(k)}_{\sigma}}{Y_{\vec s}}}.
  \end{align*}
If we can prove that last term above is equal to $A_{s_{k-i}}\cE_{\vec s,s_{k-i}}$
then we can apply \eqref{eqn:P alt AC} and finish the proof. Therefore, 
  \begin{align*}
  {Y_{\vec s}}^{2^i}\cE^{\vec s(k)}_{s_{k-i}}\Bigg|_{X_{\sigma}\leftarrow
  \frac{\cE^{\vec s(k)}_{\sigma}}{Y_{\vec s}}}
    & =\blp\prod_{r=1}^{k-i-1}A_{s_r}\brp^2\blp\prod_{\rho\notin\vec s}X_{\rho}\brp^2
      (\cE^{\vec s(k)}_{\sigma})^2\blp{Y_{\vec s}}^{2^{i-1}-1}
      \prod_{r=k-(i-2)}^kP^{\vec s(k)}_{s_r}\Bigg|_{X_{\sigma}\leftarrow
      \frac{\cE^{\vec s(k)}_{\sigma}}{Y_{\vec s}}}\brp^2 \\
    & \qquad\qquad+Y_{\vec s(k-i-1)}Y_{\vec s(k-i+1)}\mathcal{C}^{\vec s(k)}_{s_{k-i}}{Y_{\vec s}}^{2^i} \\
    & =\blp\prod_{r=1}^{k-i-1}A_{s_r}\brp^2
      \blp\prod_{\rho\notin\vec s}X_{\rho}\brp^2 
      \blp\prod_{r=k-(i-2)}^{k+1}P^{\vec s}_{s_r}\brp^2 
      +Y_{\vec s(k-i-1)}Y_{\vec s(k-i+1)}\mathcal{C}^{\vec s}_{s_{k-i}}\\
    & =\cE^{\vec s}_{s_{k-i}},
  \end{align*}
finishing the proof by induction.
\end{proof}
With this Lemma, we can now prove that $\bm{\cE}^{\vec s}$ are precisely the exchange
polynomials in the seed $\tau_n^{\vec s}$.

\begin{theorem}
\label{thm:exchange polynomials}
Suppose that $\vec s\subset[n]$ is an activation sequence of length $k$.
The exchange polynomials of $\tau_n^{\vs}$ are $\bm{\cE}^{\vs}$. 
\end{theorem}

\begin{proof}
We will prove this theorem by induction on the length $k$ of the activation sequence $\vec s$. Since
this theorem trivially holds when $k=0$ and we have our initial seed $\tau_n$, we must start by proving
that the theorem holds when $k=1$. Consider the activation sequence $(s_1)\subset[n]$. Since 
$F_{s_1}\in\mathbf{F}$ is independent of $X_{s_1}$, it follows that
$E^{(s_1)}_{s_1}=F_{s_1}=\cE^{(s_1)}_{s_1}$. Consider $\sigma\in[n]$ such that $\sigma$ and $s_1$ are distinct. Now
we must mutate $F_\sigma$ at $s_1$. Let
  $$G^{(s_1)}_{\sigma}=F_{\sigma}\Bigg|_{X_{s_1}\leftarrow\frac{\hat{F}_{s_1}|_{X_\sigma\leftarrow
    0}}{Y_{(s_1)}}}=\blp\prod_{{\rho\notin\vec (s_1)}
    \atop{\rho\neq\sigma}}X_\rho\brp\blp\frac{A_{s_1}}{Y_{(s_1)}}\brp+A_\sigma$$
which has no common factors with $A_{s_1}=F_{s_1}|_{X_\sigma\leftarrow0}$, and so we have $H^{
(s_1)}_{\sigma}=G^{(s_1)}_{\sigma}$. Clearing the denominators by multiplying $H^{(s_1)}_{\sigma}$
by the cluster variable $Y_{(s_1)}$ we get that
  $$E^{(s_1)}_{\sigma}=A_{s_1}\blp\prod_{{\rho\notin(s_1)}\atop{\rho\neq\sigma}}
    X_\rho\brp+A_\sigma Y_{(s_1)}\stackrel{\eqref{eqn:E notAC}}{=}\cE^{(s_1)}_{\sigma}.$$
    
Suppose that for all activation sequences of length $k$ that our theorem holds. Let
$\vec s=(s_1,\dots,s_{k+1})\subset [n]$ be an activation sequence of length $k+1$. Since the 
subactivation sequence $\vec s(k)$ is of length $k$, we are assuming that $\tau_n^{\vec s(k)}=
\big(\mathbf{Y}^{\vec s(k)},\bm{\cE}^{\vec s(k)}\big)$. Thus, we must show that
$\mu_{s_{k+1}}\big(\mathbf{Y}^{\vec s(k)},\bm{\cE}^{\vec s(k)}\big)
=\big(\mathbf{Y}^{\vec s},\bm{\cE}^{\vec s}\big)$ 
in order to complete the induction. First notice that $E^{\vec s}_{s_{k+1}}=\cE^{\vec s(k)}_{s_{k+1}}=
\cE^{\vec s}_{s_{k+1}}$ because $\cE^{\vec s(k)}_{s_{k+1}}$ is independent of $X_{s_{k+1}}$.
Also note the computation that $E^{\vec s}_{s_{k}}=\cE^{\vec s}_{s_{k}}$ is done in the same manner as above
in the base case.
Next, consider $\sigma\in[n]$ such that $\sigma\notin\vec s$, and then 
mutate $\cE^{\vec s(k)}_{\sigma}$ at $s_{k+1}$. Let
  $$G^{\vec s}_{\sigma}=\cE^{\vec s(k)}_{\sigma}\Bigg|_{X_{s_{k+1}}
    \leftarrow\frac{\hat{\cE}^{\vec s(k)}_{s_{k+1}}|_{X_\sigma\leftarrow0}}{Y_{\vec s}}}.$$
From Proposition~\ref{prop:denom} we know that $\hat{\cE}^{\vec s(k)}_{s_{k+1}}=\cE^{\vec s(k)}_{
s_{k+1}}$. Then 
  $$G^{\vec s}_{\sigma}=\blp\prod_{r=1}^{k}A_{s_r}\brp\blp\prod_{{\rho\notin \vec s} \atop
    {\rho\neq\sigma}}X_\rho\brp\blp\frac{A_{s_{k+1}}Y_{\vec s(k)}}{Y_{\vec s}}
    \brp+A_\sigma Y_{\vec s(k)}.$$
which has $Y_{\vec s(k)}$ as its only common factor with $A_{s_{k+1}}Y_{\vec s(k)}=
\cE^{\vec s(k)}_{s_{k+1}}|_{X_{\sigma}\leftarrow0}$ and so $H^{\vec s}_{\sigma}=G^{\vec s}_{\sigma}
/Y_{\vec s(k)}$. Clearing the denominator by multiplying by $Y_{\vec s}$ we get
  $$E^{\vec s}_{\sigma}=A_{s_{k+1}}\blp\prod_{r=1}^{k}A_{s_r}\brp\blp\prod_{{\rho\notin \vec s} \atop
    {\rho\neq\sigma}}X_\rho\brp+A_\sigma Y_{\vec s}=\cE^{\vec s}_{\sigma}.$$
Lastly, consider $s_{i}\in \vec s(k-1)$, and let
 {\small\begin{align}
 \notag
 G^{\vec s}_{s_i}
   & =\cE^{\vec s(k)}_{s_i}\Bigg|_{X_{s_{k+1}}
     \leftarrow\frac{\hat{\cE}^{\vec s(k)}_{s_{k+1}}|_{Y_{\vec s(i)}\leftarrow0}}{Y_{\vec s}}} \\
   & =\left[\blp\prod_{r=1}^{i-1}A_{s_r}\brp^2\blp\prod_{\rho\notin\vec s(k)}X_\rho\brp^2
     \blp\prod_{r=i+2}^{k} P^{\vec s(k)}_{s_r}\brp^2+A_{s_i}Y_{\vec s(i-1)}T_{\vec s(i+1)}
     \mathcal{C}^{\vec s(k)}_{s_i} 
     \right]\Bigg|_{X_{s_{k+1}}
     \leftarrow\frac{\hat{\cE}^{\vec s(k)}_{s_{k+1}}|_{Y_{\vec s(i)}\leftarrow0}}{Y_{\vec s}}}.
     \label{eqn:thm:exchange polynomial AC}
 \end{align}}
Note that there are two factors of $X_{s_{k+1}}$ on the first term and no factors in the second
term in \eqref{eqn:thm:exchange polynomial AC}. Once the 
substitution is made \eqref{eqn:thm:exchange polynomial AC} is not divisible by
$\hat{\cE}^{\vec s(k)}_{s_{k+1}}|_{Y_{\vec s(i)}\leftarrow0}$.
Since $\hat{\cE}^{\vec s(k)}_{s_{k+1}}|_{Y_{\vec s(i)}\leftarrow0}=\cE^{\vec s(k)}_{s_{k+1}}$, which
is an exchange polynomial,
by inductive hypothesis, it must satisfy {\bf LP1} and has no factors other than units and itself. 
Therefore, $G^{\vec s}_{s_i}$ has no common factors with $\hat{\cE}^{\vec s(k)}_{s_{k+1}}|_{Y_{\vec s(i)}\leftarrow0}$ and so we can let
  $$H^{\vec s}_{s_i}=G^{\vec s}_{s_i}.$$
To finish mutating $\tau_n^{\vec s(k)}$ at $s_{k+1}$ we must now clear any denominators in 
$H^{\vec s}_{s_i}$. Applying Lemma~\ref{lem:exchange polynomials} we can achieve 
this by multiplying by ${Y_{\vec s}}^{2^{k-i}}$. Thus
  {\small\begin{align*}
  E^{\vec s}_{s_i}
    & ={Y_{\vec s}}^{2^{k-i}}H^{\vec s}_{s_i} \\
    & ={Y_{\vec s}}^{2^{k-i}}
      \left[\blp\prod_{r=1}^{i-1}A_{s_r}\brp^2\blp\prod_{\rho\notin\vec s(k)}X_\rho\brp^2
      \blp\prod_{r=i+2}^{k} P^{\vec s(k)}_{s_r}\brp^2+A_{s_i}
      Y_{\vec s(i-1)}Y_{\vec s(i+1)}\mathcal{C}^{\vec s(k)}_{s_i} 
      \right]\Bigg|_{X_{s_{k+1}}
      \leftarrow\frac{\hat{\cE}^{\vec s(k)}_{s_{k+1}}|_{Y_{\vec s(i)}\leftarrow0}}{Y_{\vec s}}}. \\
    & ={Y_{\vec s}}^{2^{k-i}-1}
      \blp\prod_{r=1}^{i-1}A_{s_r}\brp^2\blp\prod_{\rho\notin\vec s}X_\rho\brp^2
      {\cE^{\vec s}_{s_{k+1}}}^2\blp\prod_{r=i+2}^k P^{\vec s(k)}_{s_r}\brp^2+A_{s_i}
      Y_{\vec s(i-1)}Y_{\vec s(i+1)}\mathcal{C}^{\vec s(k)}_{s_i}{Y_{\vec s}}^{2^{k-i}} \\
    & =\blp\prod_{r=1}^{i-1}A_{s_r}\brp^2\blp\prod_{\rho\notin\vec s}X_\rho\brp^2
      \blp\prod_{r=i+2}^{k+1} P^{\vec s}_{s_r}\brp^2+A_{s_i}
      Y_{\vec s(i-1)}Y_{\vec s(i+1)}\mathcal{C}^{\vec s}_{s_i} \\
    & =\cE^{\vec s}_{s_i},
  \end{align*}}
finishing the proof by induction.
\end{proof}

\begin{corollary}
\label{cor:exchange polynomials}
Suppose that $\vec s\subset[n]$ is an activation sequence of length $k$, then the cluster variable
$Y_{\vec s}$ in the seed $\tau_n^{\vec s}=\big(\mathbf{Y}^{\vec s},\mathbf{E}^{\vec s}\big)$
can be explicitly written in terms of the initial cluster variables
$X_\ell\in\mathbf{X}$; namely,
  \begin{equation}
  \label{eqn:cor exchange polynomials}
  Y_{\vec s}=\frac{\sum\limits_{i=1}^{k}\lb
  \blp\prod\limits_{{j=1}\atop{j\neq i}}^kA_{s_j}\brp
  \blp\prod\limits_{{j=1}\atop{j\neq s_i}}^nX_j\brp\rb 
  +\blp\prod\limits_{i=1}^{k}A_{s_i}\brp}{\blp\prod\limits_{i=1}^k X_{s_i}\brp }.
  \end{equation}
Thus, due to symmetry, $Y_{\vec s}=Y_{\vec s'}$ if and only if $s=s'$.
\end{corollary}

\begin{proof}
We will prove this by induction on $k$, the length of the activation sequence. First 
consider the case where $\vec s=(s_1)$. The cluster variable $Y_{(s_1)}$ is defined to be
$\mu_{s_1}(X_{(s_1)})$, and therefore 
\eqref{equation:mutating variables}
gives us that 
  $$Y_{(s_1)}=\frac{\hat{F}_{s_1}}{X_{s_1}}=\frac{\blp
    \prod\limits_{\rho\notin(s_1)}X_\rho\brp+A_{s_1}}{X_{s_1}}$$
which agrees with \eqref{eqn:cor exchange polynomials}. 

Now suppose that 
\eqref{eqn:cor exchange polynomials} holds for all activation sequences of length $k$. Let 
$\vec s\subset[n]$ be an activation sequence of length $k+1$. Then $Y_{\vs}$ is defined to be
$\mu_{s_k}(Y_{\vs(k-1)})$, and therefore \eqref{equation:mutating variables} and Proposition~\ref{prop:denom} 
gives us that
  $$Y_{\vec s}=\frac{\cE^{\vec s(k)}_{s_{k+1}}}{X_{s_{k+1}}}=
    \frac{\blp\prod\limits_{r=1}^{k}A_{s_r}\brp\blp\prod\limits_{\rho\notin \vec s(k)}
    X_\rho\brp+A_{s_{k+1}} Y_{\vec s(k)}}{X_{s_{k+1}}}.$$
Since $\vec s(k)$ has length $k$, we can substitute \eqref{eqn:cor exchange polynomials}
for $Y_{\vec s(k)}$ above
  \begin{align*}
  Y_{\vec s}
    & =\frac{\blp\prod\limits_{r=1}^{k}A_{s_r}\brp\blp\prod\limits_{\rho\notin \vec s}
      X_\rho\brp}{X_{s_{k+1}}}\frac{\blp\prod\limits_{\rho\in\vec s(k)} X_{\rho}\brp}{
      \blp\prod\limits_{\rho\in\vec s(k)} X_{\rho}\brp}
      +\frac{A_{s_{k+1}}}{X_{s_{k+1}}}\frac{\sum\limits_{i=1}^{k}\lb
      \blp\prod\limits_{{j=1}\atop{j\neq i}}^kA_{s_j}\brp
      \blp\prod\limits_{{j=1}\atop{j\neq s_i}}^nX_j\brp\rb 
      +\blp\prod\limits_{i=1}^{k}A_{s_i}\brp}{\blp\prod\limits_{i=1}^k X_{s_i}\brp} \\
    & =\frac{\blp\prod\limits_{{j=1}\atop{j\neq k+1}}^{k+1}A_{s_j}\brp
      \blp\prod\limits_{{j=1}\atop{j\neq s_{k+1}}}^nX_j\brp}{
      \blp\prod\limits_{i=1}^{k+1} X_{s_i}\brp}
      +\frac{\sum\limits_{i=1}^{k}\lb
      \blp\prod\limits_{{j=1}\atop{j\neq i}}^{k+1}A_{s_j}\brp
      \blp\prod\limits_{{j=1}\atop{j\neq s_i}}^nX_j\brp\rb 
      +\blp\prod\limits_{i=1}^{k+1}A_{s_i}\brp}{\blp\prod\limits_{i=1}^{k+1} X_{s_i}\brp} \\
    & =\frac{\sum\limits_{i=1}^{k+1}\lb
      \blp\prod\limits_{{j=1}\atop{j\neq i}}^{k+1}A_{s_j}\brp
      \blp\prod\limits_{{j=1}\atop{j\neq s_i}}^nX_j\brp\rb 
      +\blp\prod\limits_{i=1}^{k+1}A_{s_i}\brp}{\blp\prod\limits_{i=1}^{k+1} X_{s_i}\brp},
  \end{align*}
finishing the proof by induction.
\end{proof}

\subsection{Mutation}

Throughout this section we will use the following notation:
If $\vs\subset[n]$ is an activation sequence of length $k$, then for all $\ell\in[n]$
we will define the following activation sequence $\mu_\ell(\vs)$ by
  $$\mu_{\ell}(\vs)=\begin{cases}
    (s_1,\dots,s_k,\sigma) & \text{if $\ell=\sigma\notin\vS$,} \\
    \vs(k-1) & \text{if $\ell=s_k\in\vS$,} \\
    (s_1,\dots,s_{j-1},s_{j+1},s_{j},\dots,s_k) & \text{if $\ell=s_j\in\vS(k-1)$.}
  \end{cases}$$
This mutation action on sequences was introduced in the discussion following Theorem~\ref{thm:main theorem}, and we are now formalizing it. We want to now show that our seeds mutate properly, because as we have already mentioned, one
of the steps needed to prove Theorem~\ref{thm:main theorem} is to show that $\mu_{\ell}(\tau_n^{\vs})=\tau_n^{\mu_{\ell}(\vs)}$.
The proof of this will require a theorem of Lam and Pylyavskyy \cite[Theorem~2.4]{LP2}.

\begin{theorem}
\label{thm:mutation commute}
Let $t=(\mathbf{x},\mathbf{F})$ be a seed and $i\neq j$ be two indicies be such that 
$x_j$ does not appear in $F_i$ and such that $F_i/F_j$ is not a unit in $R$. Then the mutations at $i$ and $j$ commute. 
More precisely, we can choose seed mutations so that $$\mu_i(\mu_j(t))=\mu_j(\mu_i(t)).$$
\end{theorem}

By the recursive structures given in Definition~\ref{defn:recursive structure}, for any activation sequence $\vs\subset[n]$ of
length $k$, we have that $Y_{\vs(j)}$ does not appear
in $\cE_{s_i}^{\vs}$ if $1\leq j\leq i-2\leq k$. This will allow us to simplify our task greatly as it will allow us
to reduce the number of possible mutations we have to consider.

\begin{theorem}\label{thm:mutating seeds}
Suppose $\vs\subset[n]$ is an activation sequence of size $k$. Then 
  \begin{equation}\label{eqn:thm mutating seeds}
  \mu_\ell(\tau_n^{\vs})=\tau_n^{\mu_\ell(\vs)}
  \end{equation}
for any $\ell\in[n]$. In other words, mutating seeds in $\mathcal{A}(\tau_n)$ corresponds to mutating activation sequences.
\end{theorem}

\begin{proof}
We begin by noting that Theorem~\ref{thm:mutation commute} allows us to reduce this problem in the following way.
By definition of seeds generated by an activation sequences, we have that
  \begin{equation}\label{eqn:mutating notAC}
  \mu_\sigma(\tau_n^{\vs})=\tau_n^{\mu_\sigma(\vs)}
  \end{equation}
for all $\sigma\notin\vs$. Consider any $s_i\in\vs$. Then for all $i+2\leq j\leq k$, we have that $Y_{\vs(i)}$ does
not appear in $\ce_{s_j}^{\vs(j+1)}$, and so Theorem~\ref{thm:mutation commute} allows us to conclude that
  \begin{equation}\label{eqn:mutation commute}
  \mu_{s_i}(\tau_n^{\vs})=\mu_{s_i}\mu_{s_k}\cdots\mu_{s_{i+2}}(\tau^{\vs(i+1)})=\mu_{s_k}\cdots\mu_{s_{i+2}}\mu_{s_i}
    (\tau_n^{\vs(i+1)}).
  \end{equation}
Therefore, if we were to prove that $\mu_{s_{k-1}}(\tau_n^{\vs})=\tau^{\mu_{s_{k-1}}(\vs)}$ for any activation sequence
$\vs\subset[n]$ of length $k$, then we have also proven \eqref{eqn:thm mutating seeds} for all $\ell$ because
of \eqref{eqn:mutating notAC} and \eqref{eqn:mutation commute}. 
For simplicity we will write $\vs'=(s'_1,\dots,s'_k)=\mu_{s_{k-1}}(\vs)$ and so $s'_r=s_r$ for all $r\neq k,k-1$ and
$s'_{k-1}=s_k$ and $s'_k=s_{k-1}$. 

First we will show that $\mu_{s_{k-1}}(\mathbf{Y}^{\vs})=\mathbf{Y}^{\vs'}$.
Note that the respective underlying sets $s(j)$ and $s'(j)$ of $\vs(j)$ and $\vs'(j)$ are the same for all $j\neq k-1$.
Then we can use \eqref{equation:mutating variables} and Corollary~\ref{cor:exchange polynomials} and get that
$\mu_{s_{k-1}}(Y_{\vs(j)})=Y_{\vs'(j)}$ for all $j\neq k-1$. For $j=k-1$ we can use 
\eqref{eqn:cor exchange polynomials} and simply algebra to find that
  $$Y_{\vs'(k-1)}Y_{\vs(k-1)}-Y_{\vs(k-2)}Y_{\vs}=\blp\prod_{r=1}^{k-2}A_{s_r}\brp^2\blp\prod_{\rho\notin\vs}X_\rho
    \brp^2.$$
Therefore, $Y_{\vs'(k-1)}=\hat{\ce}_{s_{k-1}}^{\vs}/Y_{\vs(k-1)}$ which is $\mu_{s_{k-1}}(Y_{\vs(k-1)})$ by 
\eqref{equation:mutating variables}, and so $\mu_{s_{k-1}}(\mathbf{Y}^{\vs})=\mathbf{Y}^{\vs'}$.

Now we want to show that $\mu_{s_{k-1}}(\bm{\ce}^{\vs})=\bm{\ce}^{\vs'}$. Since $\ce_{s_{k}}^{\vs}$ and $\ce_{\sigma}^{\vs}$
do not contain $Y_{\vs(k-1)}$ for all $\sigma\notin\vs$, we have that $\ce_{s_{k-1}}^{\vs}$ and every $\ce_{\sigma}^{\vs}$
is unaffected by mutation at $s_{k-1}$ and so $\mu_{s_{k-1}}(\ce_{s_{k-1}}^{\vs})=\ce_{s'_{k-1}}^{\vs'}$
and $\mu_{s_{k-1}}(\ce_{\sigma}^{\vs})=\ce_\sigma^{\vs'}$. We can directly compute $\mu_{s_{k-1}}(\ce_{s_k}^{\vs})$ by looking at
  \begin{align*}
  \ce_{s_{k}}^{\vs}\Bigg|{}_{Y_{\vs(k-1)}\leftarrow \frac{\hat \ce_{s_{k-1}}^{\vs}|_{Y_{\vs}\leftarrow0}}{Y_{\vs'(k-1)}}}
    & =\blp\prod_{r=1}^{k-1}A_{s_r}\brp\blp\prod_{\rho\notin\vs}X_\rho\brp+
      A_{s_k}\blp\prod_{r=1}^{k-2}A_{s_r}\brp^2\blp\prod_{\rho\notin\vs}X_\rho\brp^2Y_{\vs'(k-1)}^{-1} \\
    & =\blp\prod_{r=1}^{k-2}A_{s_r}\brp\blp\prod_{\rho\notin\vs}X_\rho\brp\lb
      \blp\prod_{r=1}^{k-1}A_{s'_r}\brp\blp\prod_{\rho\notin\vs'}X_\rho\brp+A_{s'_k}Y_{\vs'(k-1)}\rb Y_{\vs'(k-1)}^{-1} \\
    & =\blp\prod_{r=1}^{k-2}A_{s_r}\brp\blp\prod_{\rho\notin\vs}X_\rho\brp\ce_{s'_{k-1}}^{\vs'}Y_{\vs'(k-1)}^{-1}.
  \end{align*}
Dividing out any common factors this has with 
  $$\hat\ce_{s_{k-1}}^{\vs}\Big|_{Y_{s_k}\leftarrow0}=\blp\prod_{r=1}^{k-2}A_{s_r}
    \brp^2\blp\prod_{\rho\notin\vs}X_\rho\brp^2$$
and clearing the denominator $Y_{\vs'(k-1)}$, we have that $\mu_{s_{k-1}}(\ce_{s_k}^{\vs})=\ce_{s'_{k-1}}^{\vs'}$.
We can also similarly compute  $\mu_{s_{k-1}}(\ce_{s_{k-2}}^{\vs})$ by 
looking at{\scriptsize
  \begin{align*}
  \ce_{s_{k-2}}^{\vs}\Bigg|{}_{Y_{\vs(k-1)}\leftarrow \frac{\hat \ce_{s_{k-1}}^{\vs}|_{Y_{\vs(k-2)}\leftarrow0}}{Y_{\vs'(k-1)}}}
    & =\blp\prod_{r=1}^{k-3}A_{s_r}\brp^2\blp\prod_{\rho\notin\vs}X_\rho\brp^2
      \lb\blp\prod_{r=1}^{k-2}A_{s_r}\brp\blp\prod_{\rho\notin\vs}X_\rho\brp
      +A_{s_k}\blp\prod_{r=1}^{k-2}A_{s_r}\brp^2\blp\prod_{\rho\notin\vs}X_\rho\brp Y_{\vs'(k-2)}^{-1}\rb^2 \\
    & \qquad\qquad +Y_{\vs(k-3)}\blp\prod_{r=1}^{k-2}A_{s_r}\brp^2\blp\prod_{\rho\notin\vs}X_\rho\brp Y_{\vs'(k-2)}^{-2}Y_{\vs}^2 \\
    & =\blp\prod_{r=1}^{k-2}A_{s_r}\brp^2\blp\prod_{\rho\notin\vs}X_\rho\brp^2
      \lb\blp\prod_{r=1}^{k-3}A_{s_r}\brp^2\blp\prod_{\rho\notin\vs}X_\rho\brp^2
      \lb A_{s_k}\blp\prod_{r=1}^{k-2}A_{s_r}\brp^2\blp\prod_{\rho\notin\vs}X_\rho\brp\right.\right.\\
    & \qquad\qquad+A_{s_{k-1}}Y_{\vs'(k-1)}\Bigg]^2
      +Y_{\vs(k-3)}Y_{\vs'(k-1)}Y_{\vs}^2\Bigg] Y_{\vs'(k-1)}^{-2} \\
    & =\frac{\blp\prod\limits_{r=1}^{k-2}A_{s_r}\brp^2\blp\prod\limits_{\rho\notin\vs}X_\rho\brp^2}{Y_{\vs'(k-1)}^2}\ce_{s'_{k-2}}^{\vs'}.
  \end{align*}}
Dividing out any common factors this has with $\hat{\ce}_{s_{k-1}}|_{Y_{\vs(k-2)}}$ and clearing the denominator, we have that
$\mu_{s_{k-1}}(\ce_{s_{k-1}}^{\vs})=\ce_{s'_{k-2}}^{\vs'}$.

Now we consider the cases of $\mu_{s_{k-1}}(\ce_{s_{j}}^{\vs})$ with $j\leq k-3$. We will prove the theorem by induction on $j$.
For $j=i-2$, we have that $\ce_{s_i}^{\vs}|_{Y_{\vs(j)}\leftarrow0}=\ce_{s_i}^{\vs}$ which will simplify the calculations.
We also have that
  $$P_{s_k}^{\vs}P_{s_{k-1}}^{\vs}\Bigg|{}_{Y_{\vs(k-1)}\leftarrow \frac{\hat \ce_{s_{k-1}}^{\vs}}{Y_{\vs'(k-1)}}}=
    \frac{\blp\prod\limits_{r=1}^{k-2}A_{s_r}\brp^2\blp\prod\limits_{\rho\notin\vs}X_\rho\brp^2+Y_{\vs(k-2)}Y_{\vs}}{Y_{\vs'(k-1)}^2}P_{s'_k}^{\vs'}
    P_{s'_{k-1}}^{\vs'}.$$
Then
  $$\ce_{s_{k-3}}^{\vs}\Bigg|{}_{Y_{\vs(k-1)}\leftarrow \frac{\hat \ce_{s_{k-1}}^{\vs}}{Y_{\vs'(k-1)}}}=
    \frac{\left(\blp\prod\limits_{r=1}^{k-2}A_{s_r}\brp^2\blp\prod\limits_{\rho\notin\vs}X_\rho\brp^2+Y_{\vs(k-2)}Y_{\vs}\right)^2}{Y_{\vs'(k-1)}^4}
    \ce_{s'_{k-3}}^{\vs'}$$
and so by removing common factor with $\hat{\ce}_{s_{k-1}}^{\vs}$ and clearing the denominators we have that
$\mu_{s_{k-1}}(\ce_{s_{k-3}}^{\vs})=\ce_{s'_{k-3}}^{\vs'}$. The induction argument is analogous to this one and therefore is left out.
Thus we have proven by induction that for all $j\leq k-3$, $\mu_{s_{k-1}}(\ce_{s_{j}}^{\vs})=\ce_{s'_{j}}^{\vs'}$. Hence we have shown that
\eqref{eqn:thm mutating seeds} holds for all possible values of $\ell\in[n]$.
\end{proof}

\begin{corollary}\label{cor:seed mutation}
The seeds in $\cA(\tau_n)$ are in bijection with activation sequences $\vs\subset[n]$.
Specifically, there are exactly $2^n+n-1$ cluster variables in $\cA(\tau_n)$, and there are exactly
$\sum_{k=0}^{n}n!/k!$ seeds in $\cA(\tau_n)$.
\label{corollary:combinitoric structure}
\end{corollary}
\begin{proof} Theorem~\ref{thm:mutating seeds} tells us that when you mutate a sequence generated by a activation sequence $\vs\subset[n]$, the resulting seed is generated by the mutated activation sequence. Therefore, there are $\sum_{k=0}^n\frac{n!}{k!}$
seeds in $\mathcal{A}(\tau_n)$. 
Looking at Corollary \ref{cor:exchange polynomials} the non-initial cluster variables in $\cA(\tau_n)$ are in bijection with the non-empty subsets of $[n]$, giving us $2^n + n -1$ total cluster variables. For the number of seeds, 
\end{proof}

\section{Conclusion}\label{sec:conclusion}
When we began this paper, we had the goal of proving that the combinitorial structure of our Product Laurent phenomenon algebra was the same (up to isomorphism) as the Linear Laurent phenomenon algebra's, see \cite{LP2}. Now we are going to revisit our main theorem as well as describe some conjectures and future directions for research.

\begin{maintheorem}
Let $\mathcal{A}(\tau_{n})$ be the normalized LP algebra 
generated by initial seed $\tau_{n}$. Similarly, let $\mathcal{A}(t_{n})$ 
be the normalized LP algebra generated by the initial $t_{n}$. Then then the respective exchange graphs
of $\mathcal{A}(\tau_{n})$ and $\mathcal{A}(t_{n})$ are isomorphic.
\end{maintheorem}

\begin{proof}
The proof of this Theorem lies in Corollary~\ref{cor:seed mutation}, Corollary~\ref{cor:exchange polynomials}, and
Theorem~\ref{thm:mutating seeds}. The structure of $\mathcal{A}(\tau_n)$ is that described in these three results is the same as
the structure of $\mathcal{A}(t_n)$ given in \cite{LP2}. Therefore, their exchange graphs are isomorphic.
\end{proof}

In addition to having the same exchange graphs, the product and sum cases are similar in that they both have recursive structures in their exchange polynomials. In the product case, that structure appeared in the form of our recursively defined polynomials $P_{\vec{S}(i)}$. We are particularly interested in studying LP algebras with such structure, because they can provide insight into the following conjecture posed in \cite{LP1}

\begin{conjecture}
For any finite-type LP algebra, the set of cluster monomials forms a linear basis of the subring $\mathcal{A}$. A cluster monomial is a product of powers of variables $X_1^{a_1}\dots X_n^{a_n}$ where $X_1,\dots,X_n$ are in the same seed and $a_1,\dots,a_n \in \mathbb{Z}_{\ge 0}$.
\end{conjecture}

The presence of recursiveness in exchange polynomials helps work with, and relate, variables in different clusters, so we predict that it would be easier to check this conjecture in cases like the product and sum algebras. To prove it in our product case, we hope to prove the following lemma:

\begin{conjecture} The product of two variables in different clusters can be written as the product of two variables in the same cluster plus some monomial in the base ring.
\end{conjecture}

This would allow us to swap out a product of two variables in a different cluster for a product in the same cluster. We would then hope to iterate this to prove the cluster monomials are a basis, though to do this we must prove that this swapping process terminates. Most likely, the proof will involve finding a useful grading for $\mathcal{A}$.

\end{document}